\documentclass[11pt,thmsa]{article}
\usepackage{amsmath}
\usepackage{amsfonts}
\usepackage{amssymb}
\usepackage{graphicx}
\newtheorem{theorem}{Theorem}[section]

\newtheorem{corollary}[theorem]{Corollary}

\newtheorem{definition}[theorem]{Definition}

\newtheorem{lemma}[theorem]{Lemma}

\newenvironment{proof}[1][Proof]{\noindent\textbf{#1.} }{\ \rule{0.5em}{0.5em}}

\begin{document}
\title{Willmore orbits for isometric Lie actions\thanks{Supported by NSFC (no. 11771331) and NSF of Beijing (no. 1182006)}}
\author{Ming Xu and Jifu Li\thanks{Corresponding author. E-mail: ljfanhan@126.com}}
\date{}
\maketitle

\begin{abstract}
In this work, we study the Willmore submanifolds in a closed connected Riemannian manifold which are orbits for the isometric action of a compact connected Lie group. We call them homogeneous Willmore submanifolds or Willmore orbits. The criteria for these special Willmore submanifolds
is much easier than the general theory which requires a complicated Euler-Lagrange
equation. Our main theorem claims, when the orbit type stratification for
the group action satisfies certain conditions, then we can find a Willmore
orbit in each stratified subset. Some classical examples of special importance, like Willmore torus, Veronese surface, etc., can be interpreted as Willmore orbits and easily verified with our method. Our theorems provide a large number of
new examples for Willmore submanifolds, as well as estimates for their numbers which are sharp in some classical cases.


\end{abstract}
\section{Introduction}

A Willmore submanifold $M^n$ is defined as an immersed closed submanifold in a closed connected Riemannian manifold $(N^{n+p},g_N)$, which represents a critical point for the Willmore functional $W_{N,g_N}(\cdot)$ \cite{Chen1974}\cite{Wang1998}\cite{Willmore1982}.
It has been intensively studied
in Riemannian geometry, and many examples have been found \cite{HL}\cite{Li2001}\cite{Li2002}\cite{LW2006}\cite{QTY2012}. Studying Willmore functionals and Willmore submanifolds in the unit spheres are particularly important for many reasons. Firstly, they are invariance under conformal or Mobious transformations. Secondly, they are closely related with the structure and classification theories for isoparametric hypersurfaces. And lastly, they are directly related with the famous Willmore conjecture
which has been studied for decades by many geometers
\cite{LY1982}\cite{MR1986}\cite{R1999}\cite{Willmore1965}, and finally
solved in 2014 \cite{solving-Willmore-conjecture}. Notice there is a general version for the Willmore conjecture \cite{Ko1987}\cite{Li2001}\cite{Pi1986}, which
is still widely open.

In this work, we study the Willmore functional and Willmore submanifold with Lie symmetry. We assume the ambient space admits the isometric action of a compact connected Lie group,
and study Willmore submanifolds which are orbits for this group action. We will simply call them homogeneous Willmore submanifolds or Willmore orbits.

The motivation for studying homogeneous Willmore submanifolds is two-folded.
On one hand, since the Lie method has been introduced to the study of geometric functionals, it has been very successful
on finding new examples of homogeneous minimal submanifolds \cite{HL1971}, or homogeneous Einstein metrics \cite{Bo2004}\cite{He1998}\cite{WZ1986}. So we are inspired to try similar ideas or techniques for the Willmore functional, and expect finding new examples of Willmore submanifolds. On the other hand,
we observe some classical examples of Willmore submanifolds with special importance,
like the Willmore torus and its generalizations with higher codimensions\cite{Li2001}\cite{Li2002}, the Veronese surface in $S^4(1)$
\cite{CCK1970}\cite{La1969}\cite{Li2002}, the focal submanifolds $\mathbb{C}\mathrm{P}^3$ and $\widetilde{\mathrm{G}}_2(\mathbb{R}^5)$ in $S^9(1)$ \cite{QTY2012}, and many others, permit large symmetries, i.e. they can be interpreted as Willmore orbits.

There is a simple criteria for Willmore orbits, which does not need to apply the general
Euler-Lagrange equation \cite{GLW2001}\cite{PW1988}.
First we have the orbit type stratification $N=\coprod_{\alpha\in\mathcal{A}} N_\alpha$ for
the $G$-action on $N$ according
to the isotropy group types. Then
a $G$-orbit $G\cdot x$ is a Willmore orbit iff it represents a critical point for $W_{N,g_N}(\cdot)|_{N_\alpha/G}$, where $N_\alpha$ is the stratified subset containing $G\cdot x$, and $N_\alpha/G$ is a finite dimensional smooth manifold contained in the defining space for the Willmore functional, i.e. the moduli space $\mathcal{M}_{M,N}$ of all immersions from $M=G\cdot x$ to $N$ (see Theorem \ref{theorem-1}). Computing critical points in the finite dimensional $N_\alpha/G$ is a much easier task than in the infinite dimensional $\mathcal{M}_{M,N}$.

The examples mentioned above can be easily verified to be Willmore submanifolds using this method. Further more, we expect finding new
examples. The following theorem can meet our goal in quite general sense.

\begin{theorem} \label{main-thm}
Let $G$ be a compact connected Lie group acting isometrically on a closed
connected Riemannian manifold $(N,g_N)$.
Denote $N=\coprod_{\alpha\in\mathcal{A}}N_\alpha$ the orbit type stratification for the $G$-action on $N$,
where $\mathcal{A}$ is a finite set, and $G/H_\alpha$ the orbit type for each $N_\alpha$.
Assume the following two conditions:
\begin{description}
\item{\rm (1)} $\dim H_\alpha<\dim G$ for each $\alpha$;
\item{\rm (2)} $\dim H_\alpha<\dim H_\beta$ when $N_\beta\subset\partial{N_\alpha}$.
\end{description}
Then there exists a Willmore orbit in each $N_\alpha$.
\end{theorem}

The proof of Theorem \ref{main-thm} relies totally on our criteria Theorem \ref{theorem-1} and the following theorem.

\begin{theorem}\label{collapsing-thm-1}
Let $(N,g_N)$ and $G$ be the same, and the orbit type
stratification $N=\coprod_{\alpha\in\mathcal{A}}N_\alpha$ satisfies the same
assumptions (1) and (2) as in Theorem \ref{main-thm}.
Then for any sequence $x_n\in N_\alpha$, $n\in\mathbb{N}$, such that
$\lim_{n\rightarrow\infty}x_n=x'\in N_\beta\subset\partial{N_\alpha}$, we have $$\lim_{n\rightarrow\infty}W_{N,g_N}([x_n])=\infty,$$ in which $[x_n]\in N_\alpha/G\subset
\mathcal{M}_{M,N}$, and $M$ is any $G$-orbit in $N_\alpha$.
\end{theorem}

Because Theorem \ref{collapsing-thm-1} described the behavior of
$W_{N,g_N}(\cdot)$
when the sequence of orbits $G\cdot x_n$'s collapse (converge) to one with a lower dimension,
we will simply call it the orbit collapsing theorem for the Willmore functional.

Our methods are the standard localization technique in Riemannian geometry, i.e. enlarging a sequence of infinitesimal neighborhoods to a fixed standard model space (see Section 6), together with the induction for the relative
complexity degree between $N_\alpha$ and $N_\beta$ (see Section 3 for its definition).

There are some tricky issues when we do the induction, because some extra product factor appears. Those extra product factors can not be ignored when the collapsing behavior of the orbits is complicated and we need to use the localization technique more than once.
This difficulty requires us to prove a more
general theorem than Theorem \ref{collapsing-thm-1}, which includes the extra product factor. To be precise, we will define a relative Willmore functional (See Section 5), and prove a similar orbit collapsing theorem for our new functional (see
Theorem \ref{collapsing-thm-2}). Then
Theorem \ref{collapsing-thm-1} and Theorem \ref{main-thm} are just immediate corollaries.

Imposing the assumptions (1) and (2) in
Theorem \ref{main-thm} will not loss too much generalities. We can expect
harvesting
abundant new examples of Willmore submanifolds from Theorem \ref{main-thm}. As examples, we will only discuss some important classical cases, i.e. the canonical action of $\mathrm{SO}(n_1)\times\mathrm{SO}(n_2)$ on $\mathbb{R}^{n_1+n_2}$ (and its generalization with more $\mathrm{SO}$-factors), the $SO(3)$-conjugation on $S^4(1)$, and the
$\mathrm{Ad}$-action of $SO(5)$ on $S^9(1)\subset\mathfrak{g}=\mathfrak{so}(5)$. Notice Theorem \ref{main-thm}
provides a lower bound estimate for the number of Willmore orbits, which is
irrelevant
to the equivariant changes for the metric $g_N$. For the examples we discuss, i.e. the ambient space is a standard sphere,
this estimate is sharp by either direct calculation or Theorem 5.1 in \cite{Li2001}.

Some of above examples are discussed in \cite{Li2001} as isoparametric hypersurfaces. They were further studied in \cite{QTY2012}\cite{Xie} to give more examples of
non-homogeneous Willmore submanifolds, i.e. the focal submanifolds for an isoparametric foliation in a standard sphere. The techniques in
this work can even be generalized to those situations to provide new observations for those results.

This paper is organized as following.
In Section 2, we recall some fundamental knowledge about Willmore functional and Willmore submanifold. In Section 3, we introduce the orbit type stratification. In Section 4 we define homogeneous Willmore submanifold and prove its criteria, i.e. Theorem \ref{theorem-1}. In Section 5, we define the relative Willmore functional and state the orbit collapsing theorem for it, i.e. Theorem \ref{collapsing-thm-2}. In Section 6 and Section 7, we prove Theorem \ref{collapsing-thm-2}. In Section 8, we use our new method to reconsider some classical examples.

{\bf Acknowledgement.} The authors would like to sincerely thank the referee for carefully reading this paper and providing many helpful suggestions. The authors would also like to thank Xiaobo Liu, Shaoqiang Deng, and Chao Qian for helpful discussions.

\section{Willmore functional and Willmore submanifold}

A closed connected $n$-dimensional immersed submanifold $M$ in the $(n+p)$-dimensional closed Riemannian manifold $(N,g_N)$ is called a {\it Willmore submanifold}, if it represents a critical point for the {\it Willmore functional}
\begin{equation}\label{0000}
W_{N,g_N}(M)=\int_M (S-n ||H||^2)^{n/2}d\mathrm{vol}_{M,g_N},
\end{equation}
where $S$ is the norm square of the second fundamental form of $M$ in $N$, $H$
is the mean curvature vector, and $d\mathrm{vol}_{M,g_N}$ is the volume form of $M$ defined
by the submanifold metric $g_M=i^*g_N$, where $i:M\rightarrow N$ is the immersion \cite{Li2001}.

The defining space of $W_{N,g_N}(\cdot)$ can be locally described as following.
For each pair of closed connected
manifolds $(M^n,N^{n+p})$, we define the moduli space of immersions from $M$ to $N$ as the quotient space of
all immersions $i:M\rightarrow N$, divided by all self
diffeomorphisms of $M$, and denote it as $\mathcal{M}_{M,N}$. A point in $\mathcal{M}_{M,N}$ can be either denoted as $[i]$ or $[i(M)]$,  where
$i$ is the immersion. Let $g_N$ be a given Riemannian metric on $N$. Then for each immersion $i$, the metric on $M$ is naturally induced. So (\ref{0000}) gives the definition of the Willmore functional $W_{N,g_N}(\cdot)$ for each immersion. Obviously, it is preserved by all actions of self diffeomorphisms of $M$. So $W_{N,g_N}(\cdot)$ is well defined on $\mathcal{M}_{M,N}$.

The moduli space
$\mathcal{M}_{M,N}$ is
an infinite dimensional manifold if it is not an empty or discrete set. With the metric $g_N$ chosen,
the tangent space $T_{[i]}\mathcal{M}_{M,N}$
can be identified with the space $C^\infty(i^*\mathcal{V}^{i(M),N,g_N})$ of all smooth sections of the bundle $i^*\mathcal{V}^{i(M),N,g_N}$ over $M$,
where $\mathcal{V}^{i(M),N,g_N}$ is the $g_N$-orthogonal complement sub-bundle for $i_*TM$ in $TN$.

The tangent map (i.e. the first degree differential) $DW_{N,g_N}$ of the Willmore functional is defined everywhere. An immersion $i:M\rightarrow N$ defines a Willmore submanifold iff the linear functional $DW_{N,g_N}:T_{[i]}\mathcal{M}_{M,N}\rightarrow \mathbb{R}$ vanishes identically.
This condition can be equivalently presented by the geometric data, which is
called the Euler-Lagrange equation for the Willmore functional. See
\cite{GLW2001}\cite{Li2002}\cite{PW1988} for its explicit formula when $N$ is a standard sphere. For general $N$, the Euler-Lagrange equation is very complicated.

When $n=0$ or $1$, $W_{N,g_N}(\cdot)$ is a constant functional, and when $p=0$, $\mathcal{M}_{M,N}$ is a discrete set. In this work, we count all immersed submanifolds in these special cases as Willmore submanifolds.

We have mentioned some classical examples of Willmore submanifolds with special importance in Section 1. We will revisit them in the
last section, with more details given.

\section{The orbit type stratification}

Consider a closed connected Riemannian manifold $N$ with the isometric action of a compact connected Lie group $G$. By the slice theorem in differential geometry \cite{MY1957}, we have the {\it orbit type stratification}
$N=\coprod_{\alpha\in\mathcal{A}}N_\alpha$, where $\mathcal{A}$ is a finite
set. Here we summarize some fundamental properties for this stratification.

For each $\alpha\in\mathcal{A}$, $N_\alpha$ and $N_\alpha/G$ are connected smooth manifolds. They are not closed or compact in general. All $G$-orbits in a single $N_\alpha$ has the same isotropy group up to conjugations. If $G\cdot x=G/H_\alpha$ for some $x\in N_\alpha$, we will simply call $H_\alpha$ the {\it isotropy group type} of $N_\alpha$, and call $G/H_\alpha$ the {\it orbit type} of $N_\alpha$. The subset $N_\alpha$ is a smooth fiber bundle over $N_\alpha/G$ which fibers are $G$-orbits.

The closure $\overline{N_\alpha}$ of $N_\alpha$ in $N$ is the union
of $N_\alpha$ with several other $N_\beta$ with lower dimensions. We denote
$\partial N_\alpha=\overline{N_\alpha}\backslash N_\alpha$.
For $N_\beta\subset\partial {N_\alpha}$, then we can find $x'\in N_\beta$ and $x\in N_\alpha$, such that the isotropy group $H_\beta$ at $x'$ contains the isotropy group $H_\alpha$ at $x$. Further more $H_\beta$ either has a bigger dimension, or has the same dimension but more components than $H_\alpha$.

Now we define the {\it relative complexity degree} between two stratified subsets $N_\alpha$ and $N_\beta\subset\partial N_\alpha$. It is the maximal integer $d$ such that we can insert $d$
subsets $N_{\alpha_i}$, $1\leq i\leq d$, between $N_{\alpha_0}=N_\alpha$
and $N_{\alpha_{d+1}}=N_\beta$, such that $N_{\alpha_{i+1}}\subset\partial
N_{\alpha_i}$ for $0\leq i\leq d$.

In the rest of this section, we will construct a standard local chart at
$x'\in N_\beta$,
with respect to the group action and stratification.

For any $x'\in N_\beta$, the isotropy group $H_\beta$ acts on $T_{x'}N$, which
preserves the $g_N$-orthogonal decomposition
$$T_{x'}N=T_{x'}(G\cdot x')\oplus
\mathcal{V}_{x'}^{N_\beta,N,g_N}\oplus\mathcal{V}_{x'}^{G\cdot x', N_\beta,g_N},$$
where the second and third factors in the right side are
the $g_N$-orthogonal complement subspaces at $x'$ for $N_\beta$ in $N$ and $G\cdot x'$ in $N_\beta$
respectively.
The $H_\beta$-action on the second factor will be important for our later discussions. We will simply call it the {\it isotropy action}.

We
denote $n_1=\dim G\cdot x'$,
$n_2=\dim N-\dim N_\beta$ and $n_3=\dim N_\beta-\dim G\cdot x'$.
The standard local chart is defined by the following diffeomorphism $\phi$ from $D^{n_1}(c_1)\times D^{n_2}(c_2) \times D^{n_3}(c_3)$ with some positive numbers $c_1$, $c_2$ and
$c_3$, to a neighborhood $\mathcal{U}$ of $x'$ in $N$, where
$D^{n}(c)=\{v\in\mathbb{R}^n|\,\,||v||<c \mathbb\}$ denotes the open round disk centered at the origin $o$.

Firstly, we fix a linear direct sum decomposition $\mathfrak{g}=\mathfrak{h}_\beta+\mathfrak{m}$, where $\mathfrak{g}=\mathrm{Lie}(G)$ and $\mathfrak{h}_\beta=\mathrm{Lie}(H_\beta)$.
Let $D^{n_1}(c_1)$ be an open disk around 0 in $\mathfrak{m}$ such that
$\phi_1(v)=g(v)\cdot x'$ with $v\in D^{n_1}(c_1)$ and $g(v)=\mathrm{exp}(v)\in G$, is a diffeomorphism from $D^{n_1}(c_1)$
to a neighborhood of $x'$ in $G\cdot x'$. Here $\mathrm{exp}$ is
the exponential map in Lie theory.

Secondly, because the fixed point set $\mathrm{Fix}(H_\beta)\cap N_\beta$ is a
smooth submanifold, we can
choose a smooth normal slice $\mathcal{N}$ for $\mathrm{Fix}(H_\beta)\cap(G\cdot x')$ in $\mathrm{Fix}(H_\beta)\cap N_\beta$ such that $\mathcal{N}$ and $\mathrm{Fix}(H_\beta)\cap(G\cdot x')$ intersect $g_N$-orthogonally
 at $x'$. Schur Lemma for the $H_\beta$-actions implies $\mathcal{N}$ intersects
 $g_N$-orthogonally with $G\cdot x'$ at $x'$, i.e. it is a $g_N$-orthogonal normal slice of $G\cdot x'$ in $N_\beta$.
Let $D^{n_3}(c_3)$ be a local chart on this slice with $x'$ identified with the origin $o$. Notice the isotropy group of each $v\in D^{n_3}(c_3)\subset \mathcal{N}$ is $H_\beta$.

Thirdly, we trivialize the bundle $\mathcal{V}^{N_\beta,N,g_N}$ around $x'$
in $D^{n_3}(c_3)\in\mathcal{N}$ by $n_2$ smooth sections which provide a $g_N$-orthonormal basis
at each point. Then any vector $v_2\in D^{n_2}(c_2)\subset \mathcal{V}_{x'}^{N_\beta,N,g_N}$ can also be regarded as a tangent vector with the same $g_N$-length at any point in $D^{n_3}(c_3)\in\mathcal{N}$.

Finally we can construct the diffeomorphism as the following,
\begin{equation}\label{local-model-without-parameter-space}
\phi(v_1,v_2,v_3)=g(v_1)\cdot\exp_{v_3,g_N}v_2,
\end{equation}
where $v_i\in D^{n_i}(c_i)$ with sufficiently small positive $c_i$, $i=1$, $2$, $3$.
Here $\exp_{v_3,g_N}$ is the exponential map at $v_3\in D^{n_3}(c_3)\subset N$ for the metric $g_N$.

We endow $D^{n_1}(c_1)\times D^{n_2}(c_2)\times D^{n_3}(n_3)$ with an Euclidean metric induced from $g_N$ at $x'$. Obviously the three factors are then orthogonal to each other. We define the
 $H_\beta$-action
on $D^{n_1}(c_1)$ and  $D^{n_3}(c_3)$ to be trivial,
and on $D^{n_2}(c_2)$ to be induced by the isotropy action.
Then we have the following lemma.
\begin{lemma}\label{lemma-5}
 When $x\in\mathcal{U}$ is sufficiently close to $x'$, or $c_1$, $c_2$ and $c_3$ in above construction are sufficiently close to 0, we can find some $y\in D^{n_2}(c_2)$ and $z\in D^{n_3}(c_3)$, such that $D^{n_1}(c_1)\times (H_\beta\cdot y)\times z$ is
an open neighborhood of $x$ in $G\cdot x$. Further more,  $H_\beta\cdot y$
is contained in $S^{n_2-1}(c)$ in the second factor $D^{n_2}(c_2)$ of $D^{n_1}(c_1)\times D^{n_2}(c_2)\times D^{n_3}(c_3)$, where $c$ is the $g_N$-distance from $x$ to $N_\beta$.
\end{lemma}

The proof is not hard, by the construction for the diffeomorphism $\phi$ described above.

\section{Willmore orbit and the criteria theorem}

Let $(N,g_N)$ be a closed connected Riemannian manifold, which admits
the isometric action of a compact connected Lie group $G$.
Denote $N=\coprod_{\alpha\in\mathcal{A}}N_\alpha$ the orbit type stratification,
$H_\alpha$ and $G/H_\alpha$ the isotropy group type and orbit type for each
$N_\alpha$ respectively.

Denote $M$ any $G$-orbit in some $N_\alpha$.
Then $N_\alpha/G$
can be naturally identified as a submanifold of $\mathcal{M}_{M,N}$, with
each $[x]\in N_\alpha/G$ for $x\in N_\alpha$ mapped to the element represented by the embedding $i:G\cdot x\rightarrow N$ in
$\mathcal{M}_{M,N}$.
Now we are ready to define homogeneous Willmore submanifolds.
\begin{definition}
We will call the $G$-orbit $G\cdot x\subset N_\alpha$ a homogeneous Willmore
submanifold or a Willmore orbit if $[x]\in N_\alpha/G$ is the critical point for the Willmore functional $W_{N,g_N}(\cdot)$.
\end{definition}

The following theorem is a convenient criteria for Willmore orbits.

\begin{theorem}\label{theorem-1}
Let $(N,g_N)$ be a closed connected Riemannian manifold with the
isometric group action of a compact connected Lie group $G$.
An orbit $G\cdot x$ in some $N_\alpha$ for the orbit type
stratification of $N$ is a Willmore orbit iff $[x]\in N_\alpha/G$ is
a critical point for $W(\cdot)|_{N_\alpha/G}$.
\end{theorem}

\begin{proof} The proof for the "only if" part is trivial. We only need to prove the "if" part.

The $G$-action on $N$ can be naturally transported to a $G$-action on $\mathcal{M}_{M,N}$. When $M$ is a $G$-orbit in some $N_\alpha$ , the point $[x]\in N_\alpha/G\subset \mathcal{M}_{M,N}$ for any $x\in N_\alpha$ is a fixed point for this $G$-action. We want to decompose the tangent space $T_{[x]}\mathcal{M}_{M,N}=C^\infty(\mathcal{V}^{G\cdot x,N,g_N})$ into
the sum of a sequence of finite dimensional irreducible representations for $G$. To make it precise, we first consider the Sobolev completion $\mathcal{H}=L^{2,k}(\mathcal{V}^{G\cdot x,N,g_N})$ of $T_{[x]}\mathcal{M}_{M,N}$. This is a Hilbert space with a countable basis. The action of $G$ can be naturally extended to $\mathcal{H}$. If we define the $L^{2,k}$-norm as
$$||s||_{L^{2,k}}^2=\int_M(||s||^2+||\nabla s||^2+\cdots+||\nabla^k s||^2)d\mathrm{vol}_M,$$
using the norms $||\cdot||$ on sections induced by the $G$-invariant metric $g_N$, and $G$-invariant connections
$\nabla:\Gamma(\mathcal{V}^{G\cdot x,N,g_N}\otimes\bigotimes^{i}T^*M)\rightarrow
\Gamma(\mathcal{V}^{G\cdot x,N,g_N}\otimes\bigotimes^{i+1}T^*M)$, then the action of $G$  preserves the inner product on $\mathcal{H}$. Because the $G$-action on $N$ is isometric, it preserves the Willmore
functional $W_{N,g_N}(\cdot)$ and its tangent map.
When $k$ is sufficiently big, by the Sobolev embedding theorem, the tangent map $DW_{N,g_N}:T_{[x]}\mathcal{M}_{M,N}\rightarrow \mathbb{R}$ of the Willmore functional can be extended to $\mathcal{H}$. So $DW_{N,g_N}:\mathcal{H}\rightarrow\mathbb{R}$ is a $G$-equivariant linear
functional as well.

The trivial representations of $G$ in $\mathcal{H}$ corresponds to all the $G$-invariant sections of $\mathcal{V}^{G\cdot x,N,g_N}$. All these sections give a finite
dimensional subspace in $T_{[x]}\mathcal{M}_{M,N}$ which can be identified
as the tangent space $T_{[x]}(N_\alpha/G)$. So we can orthogonally decompose  $\mathcal{H}$ as $$\mathcal{H}=T_{[x]}(N_\alpha/G)\oplus\mathcal{H}_1\oplus\mathcal{H}_2\oplus\cdots,$$
where each $\mathcal{H}_i$ is a finite dimensional nontrivial irreducible $G$-representation space.

By the assumption that $[x]$ is a critical point for $W_{N,g_N}(\cdot)|_{N_\alpha/G}$,
we have $DW_{N,g_N}=$ $0$ on $T_{[x]}(N_\alpha/G)$. For each $\mathcal{H}_i$, $DW_{N,g_N}(\mathcal{H}_i)=0$
by Schur Lemma. To summarize, $DW_{N,g_N}$ at $[x]$ vanishes identically on $\mathcal{H}$. So $[x]$ is a critical point for the Willmore functional, i.e. $G\cdot x$ is a Willmore orbit.
\end{proof}

An immediate consequence of Theorem \ref{theorem-1} is the following corollary.

\begin{corollary}\label{corollary-1}
Let $(N,g_N)$ be a closed connected Riemannian manifold with the
isometric group action of a compact connected Lie group $G$.
Assume some subset $N_\alpha$ for the orbit type stratification is closed in $N$, then there exists a
Willmore orbit in $N_\alpha$. If $N_\alpha$ is not a single $G$-orbit, it contains
at least two Willmore orbits.
\end{corollary}

If $N_\alpha$ is a $G$-orbit, then it is obviously a Willmore orbit by Theorem \ref{theorem-1}. If $N_\alpha$ is closed and has more than one orbit, we can find two Willmore orbits in $N_\alpha$ for the minimum and maximum points
of $W_{N,g_N}(\cdot)|_{N_\alpha/G}$.

\section{The orbit collapsing theorem for the relative Willmore functional}

We first define the relative Willmore functional.

Let $N^{n+p}$ be an immersed or embedded submanifold of a Riemannian manifold $(N',g_{N'})$. Then we have the submanifold metric $g_N$ on $N$ induced by $g_{N'}$. Generally speaking we do not need particular assumptions for $N'$, but we need $N$ to be connected and closed. The closeness of $N$ implies the compactness, and then a finite orbit type stratification.
Suppose $M'$ is an immersed or embedded submanifold of $N'$ which intersects with $N$ at a closed connected submanifold $M^n$.

Let $\mathrm{II}_{M',g_{N'}}$ be the second fundamental form of $M'$ in $(N',g_{N'})$, and $\mathrm{II}_{M,M',g_{N'}}$ the restriction of
$\mathrm{II}_{M',g_{N'}}$ to $M$. Using the metric $g_M$ on $M$ induced by $g_{N'}$ and the relative second fundamental form $\mathrm{II}_{M,M',g_{N'}}$, we can similarly define $S_{M,M',g_{N'}}$, $H_{M,M',g_{N'}}$ and $d\mathrm{vol}_{M,g_{N'}}$ as for the Willmore functional. Then the functional defined by
$$W_{N,N',g_{N'}}(M')=\int_{M}
(S_{M,M',g_{N'}}-n||H_{M,M',g_{N'}}||^2)^{n/2}d\mathrm{vol}_{M,g_{N'}}$$
is called the {\it relative Willmore functional}.

We have the following obvious lemmas for the relative Willmore functional.

\begin{lemma}\label{lemma-3}
when $(N',g_{N'})$ is the Riemannian product of $(N,g_N)$ with any
manifold $(D,g_D)$, $N$ is identified with $N\times o$ for any fixed $o\in D$, and $M'=M\times D$ for a closed connected immersed submanifold $M$ in $N$, then we have
$W_{N,N',g_{N'}}(M')=W_{N,g_N}(M)$.
\end{lemma}

\begin{lemma}\label{lemma-4}
For any constant $c>0$, we have
$W_{N,N',g_{N'}}(M')=W_{N,N',cg_{N'}}(M')$.
\end{lemma}

The orbit collapsing theorem for relative Willmore functional is the following.

\begin{theorem}\label{collapsing-thm-2}
Let $(N,g_N)$ and $G$ be the same as in Theorem \ref{main-thm}, and the orbit type
stratification $N=\coprod_{\alpha\in\mathcal{A}}N_\alpha$ satisfy the same
assumptions (1) and (2).
Let $D=D_1\times D_2\subset \mathbb{R}^{q_1}\times \mathbb{R}^{q_2}$ be any open set in $\mathbb{R}^q$ endowed with the standard
Euclidean metric $g_D=g_{D_1}\oplus g_{D_2}$. Denote
$g_{N',0}=g_N\oplus g_D$ the product metric on $N'=N\times D$.
Then for
\begin{description}
\item{\rm (1)} Any sequence of Riemannian metrics $g_{N',n}$ on $N'$, which
converge to $g_{N',0}$ in the $C^\infty$-topology,
\item{\rm (2)} Any sequence $x_n$ in the
same stratified subset $N_\alpha$, such that $\lim_{n\rightarrow\infty}x_n=x'\in N_\beta\subset \partial N_\alpha$,
\item{\rm (3)} Any sequence $z_n$ in $D$ with $\lim z_n=z\in D$,
\end{description}
we have
$$\lim_{n\rightarrow\infty}W_{N_n,N',g_{N',n}}(M'_n)=\infty,$$
where $N_n=N\times z_n$ and $M'_n=(G\cdot x_n)\times D_1$.
\end{theorem}

The proof of Theorem \ref{collapsing-thm-2} will occupy the next two sections.
Here we only use it to prove Theorem \ref{collapsing-thm-1} and Theorem \ref{main-thm}.
\bigskip

{\noindent\textbf{Proof of Theorem \ref{collapsing-thm-1}} }
Take $D=D_1$, a constant family of metrics $g_{N',n}= g_{N',0}$, and a constant sequence $z_n=z$,
then Theorem \ref{collapsing-thm-1} follows Lemma \ref{lemma-3} and
Theorem \ref{collapsing-thm-2} immediately.
{\ \rule{0.5em}{0.5em}}
\bigskip

{\noindent\textbf{Proof of Theorem \ref{main-thm}} }
Theorem \ref{collapsing-thm-1} implies $W_{N,g_N}(\cdot)|_{N_\alpha/G}$ has a
minimum point $[x]$ for some $x\in N_\alpha$. Then using Theorem \ref{theorem-1}, $G\cdot x$ is a Willmore orbit.
{\ \rule{0.5em}{0.5em}}

\section{Proof of Theorem \ref{collapsing-thm-2}}

The proof of Theorem \ref{collapsing-thm-2} is an induction for the relative complexity degree between $N_\alpha$ and $N_\beta$.

Let $(N,g_N)$ be a closed connected manifold admitting
the isometric action of a compact connected Lie group $G$.
Denote
all the subsets for its orbit type stratification as $N_\alpha$, $\alpha\in\mathcal{A}$, and assume they satisfy the conditions (1) and
(2) in Theorem \ref{collapsing-thm-2}.

We will first prove Theorem \ref{collapsing-thm-2} when the relative
complexity degree between $N_\alpha$ and $N_\beta$ is 0, which will occupy
the rest of this section.

Denote $n_1=\dim G\cdot x'$, $n_2=\dim N-\dim N_\beta$, $n_3=\dim N_\beta-\dim G\cdot x' $. We may assume $n$ is sufficiently big, such that for each $n$,
we can find a $x'_n\in N_\beta$, such that $$\lim_{n\rightarrow\infty}x'_n=x'\quad\mbox{and}\quad d(x_n,x'_n)=d(x_n,N_\beta)=r_n\rightarrow 0\mbox{ when }n\rightarrow\infty,$$
where the distance function on $N$ is induced by $g_N$.
Then we have
$$x_n\in
\exp_{x'_n,g_N} \mathcal{V}_{x'_n}^{N_\beta,N,g_N},$$ where the exponential map and the normal subspace at $x'_n$ is defined with respect to $g_N$.
As we have shown in Section 3, we can construct a diffeomorphism $\phi$ as in (\ref{local-model-without-parameter-space}) from
an open neighborhood $\mathcal{U}$ of $(x',z)$ in $N'=N\times D$, to
the product $$D^{n_1}(c_1)\times D^{n_2}(c_2)\times (D^{n_3}(c_3)\times D^{q_1}(c_4)\times D^{q_2}(c_5)),$$
using
the product metric $g_{N',0}$, where the open disks $D^{n_i}(c_i)$'s correspond to neighborhood of $(x',z)$ in the $G\cdot x'$ direction, the $g_N$-normal direction of $N_\beta$ in $N$ (which is also the $g_{N',0}$-normal direction of $N_\beta\times D$ in $N\times D$), and the normal direction of $(G\cdot x')\times z$ in $N_\beta\times D_1\times D_2$ respectively.

When $n$ is sufficiently large,
the points $x_n$ can be chosen from $\mathcal{U}$ such that
$\phi^{-1}(x_n,z_n)$ $=(y'_{n},y''_{n},y'''_{n})$ with $y'_n$, $y''_n$ and $y'''_n$ all converging to the origin. Then we also have
$\phi^{-1}(x'_n,z_n)=(y'_n,o,y'''_n)$.
Define the metrics $g'_{N',n}=r_n^{-1}g_{N',n}$ on $N\times D$, and diffeomorphisms
$$\phi_n(v_1,v_2,v_3)=\phi(r_n v_1 +y'_n,r_n v_2, r_n v_3 + y'''_n).$$
Then for sufficiently large $n$, we can find a suitable neighborhood $\mathcal{U}_n$ of $(x'_n,z_n)$ in $N\times D$, such that
$\phi_n^{-1}(\mathcal{U}_n)=D^{n_1}(1)\times D^{n_2}(2)\times (D^{n_3}(1)\times D^{q_1}(1)\times D^{q_2}(1))$, such that $\phi_n$ maps the origin to $(x'_n,z_n)$, and
the metrics $\tilde{g}_{\tilde{N}',n}=\phi_n^* g'_{N',n}=r_n^{-1}\phi_n^* g_{N',n}$ on $\tilde{N}'=D^{n_1}(1)\times D^{n_2}(2)\times (D^{n_3}(1)\times D^{q_1}(1)\times D^{q_2})$ converges in the $C^\infty$-topology, to the product metric $\tilde{g}_{N',0}$ of the
standard Euclidean metrics on each factor.

Let $H_{\beta,o}$ be the identity component of $H_\beta$, and $y_n=r_n^{-1}y''_n$. Then by Lemma \ref{lemma-5}, $H_{\beta,o}\cdot y_n\subset S^{n_2-1}(1)$. Here we summarize some notations for the standard model space
$\tilde{N}'$,
\begin{eqnarray*}
\tilde{N}'&=& D^{n_1}(1)\times D^{n_2}(2)\times (D^{n_3}(1)\times D^{q_1}(1)\times D^{q_2}(1)),\\
\tilde{S}'&=& D^{n_1}(1)\times S^{n_2-1}(1)\times
(D^{n_3}(1)\times D^{q_1}(1)\times D^{q_2}(1)),\\
\tilde{N}_n &=& D^{n_1}(1)\times D^{n_2}(2)\times (D^{n_3}(1)\times o),\\
\tilde{M}'_n&=& D^{n_1}(1)\times(H_{\beta,o}\cdot y_n)\times (o\times D^{q_1}(1)\times o)\subset \tilde{S}',\\
\tilde{M}_n &=& D^{n_1}(1)\times(H_{\beta,o}\cdot y_n)\times o=\tilde{N}_n\cap\tilde{M}'_n,\\
\tilde{g}_{\tilde{N}',n} &=& \phi_n^* g'_{N',n},\mbox{ and}\\
\tilde{g}_{\tilde{N}',0} &=& \lim_{n\rightarrow\infty}\tilde{g}_{\tilde{N}',n}.\\
\end{eqnarray*}
Notice that the $\phi_n$-images of $\tilde{M}'_n$, $\tilde{N}_n$ and $\tilde{M}_n$ are some neighborhoods of $(x_n,z_n)$ in $M'_n=(G\cdot x_n)\times D_1$,
$N_n=N\times z_n$ and $M_n=(G\cdot x_n)\times z_n$ respectively. By the assumption (2) in Theorem
\ref{collapsing-thm-2}, $n_4=\dim G\cdot x_n>n_1=\dim G\cdot x'$, so $\dim {H}_{\beta,o}\cdot y_n=n_4-n_1>0$. The metric
$\tilde{g}_{\tilde{N}',0}$ is the product of standard Euclidean metrics on all factors.

Now we prove for any $n$ sufficiently big,
\begin{equation}\label{0008}
\int_{\tilde{M}_n}
(S_{\tilde{M}_n,\tilde{M}'_n,\tilde{g}_{\tilde{N}',n}}-n_4
||H_{\tilde{M}_n,\tilde{M}'_n,\tilde{g}_{\tilde{N}',n}}||^2)^{n_4/2}
d\mathrm{vol}_{\tilde{M}_n,\tilde{g}_{\tilde{N}',n}}>C,
\end{equation}
where $C$ is a positive universal constant.

Let $\mathbf{n}_n$ and $\mathbf{n}$ be the outgoing unit normal vector fields
of $\tilde{S}'$ in
$\tilde{N}'$ with respect to
$\tilde{g}_{\tilde{N}',n}$ and $\tilde{g}_{\tilde{N}',0}$ respectively.
Because $\tilde{g}_{\tilde{N}',n}$ converges to $\tilde{g}_{\tilde{N}',0}$
in the $C^\infty$-topology,
$\mathbf{n}_n$ also converges to $\mathbf{n}$ in the $C^\infty$-topology.
The principal
curvature of $\tilde{M}'_n$ in $\tilde{N}'$, for the normal vector field $\mathbf{n}$, with respect to $\tilde{g}_{\tilde{N}',0}$, equals 1 in the directions of the
$H_{\beta,o}\cdot y_n\subset S^{n_2-1}(1)$, and equals 0 in the directions of $D^{n_1}(1)$. Notice that $n_1>0$ by the assumption (1) of the theorem. So when
we change $\tilde{g}_{\tilde{N}',0}$ and $\mathbf{n}$ to $\tilde{g}_{\tilde{N}',n}$ and $\mathbf{n}_n$ respectively, those principal curvatures are still close to 1 and 0 respectively.

Assume $\lambda_{1,n}\leq\cdots\leq\lambda_{n_4+q_2,n}$ are all the eigenvalues
of $\mathrm{II}_{\tilde{M}_n,\tilde{M}'_n,\tilde{g}_{\tilde{N}',n}}(\mathbf{n}_n)$, i.e. the restriction of the second fundamental form of $\tilde{M}'_n$ to $\tilde{M}_n$, with respect to any local $\tilde{g}_{\tilde{N}',n}$-orthonormal tangent frame of $\tilde{M}_n$ and the normal vector field $\mathbf{n}_n$. Then when $n$ is sufficiently big, we have $\lambda_{1,n}\leq1/3$ and $\lambda_{n_4,n}\geq 2/3$. There exists a universal constant $C_1>0$, such that
\begin{eqnarray}\label{0010}
S_{\tilde{M}_n,\tilde{M}'_n,\tilde{g}_{\tilde{N}',n}}-
n_4||H_{\tilde{M}_n,\tilde{M}'_n,\tilde{g}_{\tilde{N}',n}}||^2
\geq
\sum_{i=1}^{n_4}\lambda_{i,n}^2-\frac{1}{n_4}
(\sum_{i=1}^{n_4}\lambda_{i,n})^2>C_1.
\end{eqnarray}

Notice that all $H_{\beta,o}\cdot y_n$ belong to the same subset for the orbit type stratification of the $H_{\beta,o}$-action on $S^{n_2-1}(1)$. The
orbits $H_{\beta,o}\cdot y_n$ in $S^{n_2-1}(1)$ can not have any collapsing subsequence, otherwise we can insert a stratified subset $N_\gamma$ between
$N_\alpha$ and $N_\beta$ which is a contradiction that the relative complexity degree between $N_\alpha$ and $N_\beta$ is 0.
So there exists a positive lower bound for their
volumes, with respect to the metric $\tilde{g}_{\tilde{N}',0}$.
By the approximation of metrics, we also have
\begin{equation}\label{0011}
\mathrm{vol}(\tilde{M}_n, \tilde{g}_{\tilde{N}',n})
>C_2>0,
\end{equation}
where $C_2$ is universal constant.

Summarize (\ref{0010}) and (\ref{0011}), we get
\begin{equation}\label{0099}
\int_{\tilde{M}_n}
(S_{\tilde{M}_n,\tilde{M}'_n,\tilde{g}_{\tilde{N}',n}}
-n_4||H_{\tilde{M}_n,\tilde{M}'_n,\tilde{g}_{\tilde{N}',n}}||^2)^{n_4/2}
d\mathrm{vol}_{\tilde{M}_n,\tilde{g}_{\tilde{N}',n}}> C=
C_1^{n_4/2}C_2>0,
\end{equation}
which proves (\ref{0008})

To finish the proof of theorem when the relative complexity degree between $N_\alpha$ and $N_\beta$ is 0, we only need to
observe that $\mathcal{U}_n$ gets infinitely tiny when $n$ goes to infinity.
For each $n$, we denote $l(n)$ the maximal number of different group elements $g_i\in G$, $1\leq i\leq l(n)$, such that all $(g_i\cdot\mathcal{U}_n)\cap M$'s are disjoint with each other. Obviously we have $\lim_{n\rightarrow\infty}l(n)=+\infty$. When $n$ is sufficiently large, we can use similar argument as above and the uniform approximation of the metrics near $N$ to get the estimate
(\ref{0099}) with $\mathcal{U}_n$ changed to $g_i\cdot\mathcal{U}_n$ and $g_i\cdot\tilde{M}_n$ for each $i$. Then we can
 use Lemma \ref{lemma-4} and observe
\begin{eqnarray}\label{0009}
 W_{N_n,N',g_{N',n}}((G\cdot x_n)\times D)=W_{N_n,N',r_n^{-1}g_{N',n}}(
(G\cdot x_n)\times D) 
> l(n)\cdot C\rightarrow+\infty.
\end{eqnarray}

This ends the proof of Theorem \ref{collapsing-thm-2}.

In the case that $N$ is of cohomogeneity one with respect to the $G$-action, the argument in this section is enough to prove Thereom \ref{collapsing-thm-2}
and then Theorem \ref{collapsing-thm-1} and Theorem \ref{main-thm}, i.e. there exists a principal Willmore orbit.
An interesting observation is that similar techniques can be applied to prove the existence of a Willmore isoparametric
hypersurface when the ambient space is a standard sphere.
As this is not a new result \cite{Li2001}, and
not very relevant to the main theme of this work. We will
discuss it in detail.

\section{Proof of Theorem \ref{collapsing-thm-2} continued}

We continue to prove Theorem \ref{collapsing-thm-2} with the assumption that the theorem is valid when the relative complexity degree between $N_\alpha$ and $N_\beta$ is smaller than $d$.

Now we assume the relative complexity degree between $N_\alpha$ and $N_\beta$
is $d$.
We can apply exactly the same argument in the last section to reduce our
discussion to the standard model $\tilde{N}'$,
until we reach the step to prove (\ref{0008}), i.e.
\begin{equation*}
\int_{\tilde{M}_n}(S_{\tilde{M}_n,\tilde{M}'_n,\tilde{g}_{\tilde{N}',n}}
-n_4
||H_{\tilde{M}_n,\tilde{M}'_n,\tilde{g}_{\tilde{N}',n}}||^2
)^{n_4/2}d\mathrm{vol}_{\tilde{M}_n,\tilde{g}_{\tilde{N}',n}}>C>0,
\end{equation*}
where $C$ is a positive universal constant. The reason is that
when $H_{\beta,o}\cdot y_n$ collapse to an orbit with a low dimension, we can not get a positive lower bound estimate for their volumes.

Passing to suitable subsequences, we may assume $y_n\subset  S^{n_2-1}(1)$ converges to
$y\in S^{n_2-1}(1)$. If $H_{\beta,o}\cdot y$ has the same
dimension as $H_{\beta,o}\cdot y_n$, then (\ref{0010}) and (\ref{0011})
are still valid, from which we can get (\ref{0008}).
So we only need to consider the case that
$\dim H_{\beta,o}\cdot y<\dim H_{\beta,o}\cdot y_n$.

We choose the local normal vector fields $\mathbf{n}_{i,n}$, $1\leq i\leq n_1+n_2+n_3+q_2-n_4$ on $\tilde{M}'_n$, satisfying the following conditions:
\begin{description}
\item{\rm (1)} $\mathbf{n}_{1,n}=\mathbf{n}_{n}$, and other $\mathbf{n}_{i,n}$'s are tangent to $\tilde{S}'$.
\item{\rm (2)} For each $n$ and at each point, all $\mathbf{n}_{i,n}$'s provide an $\tilde{g}_{\tilde{N}',n}$-orthonormal basis for the fiber of $\mathcal{V}^{\tilde{M}'_n,\tilde{N}',\tilde{g}_{\tilde{N}',n}}$.
\end{description}

The term
$$S_{\tilde{M}_n,\tilde{M}'_n,\tilde{g}_{\tilde{N}',n}}-n_4
||H_{\tilde{M}_n,\tilde{M}'_n,\tilde{g}_{\tilde{N}',n}}||^2$$
can be locally presented using $\mathbf{n}_{i,n}$. It is the sum of two parts. The first part
$P_{1,n}$ contains all the terms for $\mathbf{n}_{1,n}=\mathbf{n}_n$,
and $P_{2,n}$ contains all the terms for $\mathbf{n}_{i,n}$ with $i>1$.
Both $P_{1,n}$ and $P_{2,n}$ are non-negative functions. In particular,
because (\ref{0010}) is still valid when $n$ is sufficiently big, i.e. $P_{1,n}>C_1>0$,
then we have
\begin{eqnarray}
& &(S_{\tilde{M}_n,\tilde{M}'_n,\tilde{g}_{\tilde{N}',n}}-n_4
||H_{\tilde{M}_n,\tilde{M}'_n,\tilde{g}_{\tilde{N}',n}}||^2)^{n_4/2}\nonumber\\
&=& (P_{1,n}+P_{2,n})^{n_4/2}
\geq C_2 P_{1,n}^{n_1/2}P_{2,n}^{(n_4-n_1)/2}
> C_1^{n_1/2}C_2 P_{2,n}^{(n_4-n_1)/2},\label{0016}
\end{eqnarray}
where $C_2$ is a positive universal constant.

Fix any $v_1\in D^{n_1}(1)$ and $v_2\in D^{q_1}(1)$, we denote
\begin{eqnarray*}
\hat{{N}}'&=& \tilde{S}'=D^{n_1}(1)\times S^{n_2-1}(1)\times (D^{n_3}(1)\times D^{q_1}(1)\times D^{q_2}(1)),\\
\hat{{N}} &=& v_1\times S^{n_2-1}(1)\times v_2,\\
\hat{{D}} &=& \hat{D}_1\times \hat{D}_2= (D^{n_1}(1)\times D^{q_1}(1))\times D^{q_2}(1),\\
\hat{M}_n &=& v_1\times (H_{\beta,o}\cdot y_n)\times v_2,\\
\hat{{M}}'_n &=& \tilde{M}'_n= \hat{M}_n\times \hat{D}_1=
D^{n_1}(1) \times (H_{\beta,o}\cdot y_n)\times (o \times D^{q_1}(1)\times o),\\
\hat{g}_{\hat{N}',n}&=&\tilde{g}_{\tilde{N}',n}|_{\hat{N}'},\\
\hat{g}_{\hat{N}',0}&=&\lim_{n\rightarrow\infty}\hat{g}_{\hat{N}',n}
\mbox{ in the }C^\infty\mbox{-topology}.
\end{eqnarray*}
Notice $\hat{g}_{\hat{N}',0}$ is just the Riemannian product of the standard metrics for all factors. The standard metric on the sphere factor is $H_{\beta,o}$-equivariant.

The complexity degree between the
subset containing all $H_{\beta,o}\cdot y_n$ and that containing $H_{\beta,o}\cdot y$, for the orbit type stratification of the $H_{\beta,o}$-action on $S^{n_2-1}(1)$ is at most $d-1$, where $d$ is the relative complexity between $N_\alpha$ and $N_\beta$.
Using the inductive assumption for $\hat{D}=\hat{D}_1\times \hat{D}_2$, $\hat{N}'$, $\hat{N}$ and $\hat{M}'_n$, for any fixed $v_1\in D^{n_1}(1)$ and
$v_2\in D^{q_1}(1)$, the corresponding $W_{\hat{N},\hat{N}',\hat{g}_{\hat{N}',n}}(\hat{M}'_n)$, i.e.
\begin{eqnarray}\label{0012}
\int_{\hat{M}_n}(S_{\hat{M}_n,\hat{M}'_n,\hat{g}_{\hat{N}',n}}
-(n_4-n_1)||H_{\hat{M}_n,\hat{M}'_n,\hat{g}_{\hat{N}',n}}||^2)^{(n_4-n_1)/2}
d\mathrm{vol}_{\hat{M}_n,\hat{g}_{\hat{N},n}}
\end{eqnarray}
goes to infinity.

Comparing the term $$\hat{P}_n=S_{\hat{M}_n,\hat{M}'_n,\hat{g}_{\hat{N}',n}}
-(n_4-n_1)||H_{\hat{M}_n,\hat{M}'_n,\hat{g}_{\hat{N}',n}}||^2$$ in (\ref{0012}) and
the term $P_{2,n}$ in
$$S_{\tilde{M}_n,\tilde{M}'_n,\tilde{g}_{\tilde{N}',n}}-n_4
||H_{\tilde{M}_n,\tilde{M}'_n,\tilde{g}_{\tilde{N}',n}}||^2,$$
containing all the terms for $\mathbf{n}_{i,n}$ with $i>1$,
we claim
\begin{lemma}\label{lemma-2} For any $v_1\in D^{n_1}(1)$ and at any point
of $\hat{M}_n$, we have
$\hat{P}\leq P_{2,n}$.
\end{lemma}
\begin{proof}
Using the local orthogonal tangent frame $\{e_i,1\leq i\leq n_4\}$ for $\hat{M}'_n=\tilde{M}'_n=D^{n_1}\times (H_{\beta,o}\cdot y_n)\times (o\times D^{q_1}(1)\times o$ in $\hat{N}'=\tilde{S}'$, with respect to $\tilde{g}_{\tilde{N}',n}$,
such that when restricted to $\hat{M}_n=v_1\times (H_{\beta,o}\cdot y_n)\times v_2$, all $e_i$'s with $i\leq n_4-n_1$ provide a local tangent frame of
$\hat{M}_n$. At each point of
$\hat{M}_n$, we can use this tangent frame to present $\mathrm{II}_{\tilde{M}_n,\tilde{M}'_n,\tilde{g}_{\tilde{N}',n}}(\mathbf{n}_{i,n})$ with $i>1$ as an $n_4\times n_4$ symmetric matrix $A_{i,n}$. Then its upper left
$(n_4-n_1)\times(n_4-n_1)$-block $B_{i,n}$ is the matrix for
$\mathrm{II}_{\hat{M}_n,\hat{M}'_n,\hat{g}_{\hat{N}',n}}(\mathbf{n}_{i,n})$. So
\begin{eqnarray}
P_{2,n}&=& \sum_{i>1}||A_{i,n}-\frac{\mathrm{tr}A_{i,n}}{n_4}I||^2
=\sum_{i>1} \min_{c\in\mathbb{R}}
||A_{i,n}-cI||^2,\label{0013}\\
\hat{P}_n&=&\sum_{i>1}||B_{i,n}-\frac{\mathrm{tr}B_{i,n}}{n_4-n_1}I||^2
=\sum_{i>1}\min_{c\in\mathbb{R}}
||B_{i,n}-cI||^2,\label{0014}
\end{eqnarray}
where the norm square for a real matrix is the sum of the squares of all entries. Using the relation between $A_{i,n}$ and $B_{i,n}$ for each $i>1$ and $n$ to compare (\ref{0013}) and (\ref{0014}), we can easily get $\tilde{P}_n\leq
P_{2,n}$ at each point of $\hat{M}_n$, for any $v_1\in D^{n_1}(1)$ and $v_2\in D^{q_1}(1)$.
\end{proof}

Now we continue our estimate for (\ref{0008}). When $n$ is
sufficiently large,
\begin{eqnarray*}
& &\int_{\tilde{M}_n}(S_{\tilde{M}_n,\tilde{M}'_n,\tilde{g}_{\tilde{N}',n}}
-n_4 ||H_{\tilde{M}'_n,\tilde{M}_n,\tilde{g}_{\tilde{N}',n}}||^2)^{n_4/2}
d\mathrm{vol}_{\tilde{M}_n,\tilde{g}_{\tilde{N}',n}}
\\&\geq&
C_3\int_{v_1\in D^{n_1}(1)}\int_{v_2\in D^{q_1}(1)}
\left[\int_{\hat{M}_n} P_{2,n}^{(n_4-n_1)/2}d\mathrm{vol}_{\hat{M}_n,\hat{g}_{\hat{N}',n}}\right]
\\&\geq&
C_3\int_{v_1\in D^{n_1}(1)}\int_{v_2\in D^{q_1}(1)} W_{\hat{N},\hat{N}',\hat{g}_{\hat{N}',n}}(\hat{M}'_n),
\end{eqnarray*}
where $C_3$ is a positive universal constant.

By the inductive assumption, for each pair of $v_1\in D^{n_1}(1)$ and $v_2\in D^{q_1}(1)$, the non-negative continuous function $$f_n(v_1,v_2)=W_{\hat{N},\hat{N}',\hat{g}_{\hat{N}',n}}(\hat{M}'_n)
=W_{\hat{N},\hat{N}',\hat{g}_{\hat{N}',n}}((H_{\beta,o}\cdot y_n)\times \hat{D}_1)$$ goes to infinity with $n$. So
$$\lim_{n\rightarrow\infty}\int_{v_1\in D^{n_1}(1)}\int_{v_2\in D^{q_1}(1)}f_n(v_1,v_2)=+\infty,$$
which is enough for us to get
$$\lim_{n\rightarrow\infty}W_{N_n,N',g_{N',n}}((G\cdot x_n)\times D_1)=+\infty.$$

To summarize, when the orbit collapsing behavior is more complicated than the case that the relative complexity degree is zero, its relative Willmore functional goes to infinity even faster.
This ends the proof of Theorem \ref{collapsing-thm-2}, when the relative complexity between $N_\alpha$ and $N_\beta$ is $d$.
So by induction, the proof of Theorem \ref{collapsing-thm-2} is done.

\section{Some classical examples revisited}
\label{classical examples section}
In this section, we revisit
some important examples of Willmore submanifolds.

{\bf Example 1.}
The {\it Willmore torus} is the Willmore hypersurface $S^{n_1}(\sqrt{\frac{n_2}{n}})\times S^{n_2}(\sqrt{\frac{n_1}{n}})$ in
$S^{n+1}$, where $n_1+n_2=n$ \cite{Li2001}. It is a principle orbit
for the canonical action of
$G=\mathrm{SO}(n_1+1)\times\mathrm{SO}(n_2+1)$ on $S^{n+1}(1)$. This group action is of cohomogeneity one, and the assumptions in Theorem \ref{main-thm} are satisfied. The explicit calculation (see Example 4 below) shows it is the only principal
Willmore orbit predicted by Theorem \ref{main-thm}.

{\bf Example 2.}
The {\it Veronese surface} is a $\mathbb{R}\mathrm{P}^2$
imbedded in $S^4(1)\subset\mathbb{R}^5$, which is a Willmore submanifold parametrized by $(x,y,z)\in\mathbb{R}^3$ with $x^2+y^2+z^2=3$ \cite{CCK1970}\cite{La1969}\cite{Li2002}, i.e.
\begin{eqnarray*}
& & u_1=\frac{1}{\sqrt{3}}yz,\quad u_2=\frac{1}{\sqrt{3}}xz,
\quad u_3=\frac{1}{\sqrt{3}}xy,\\
& & u_4=\frac{1}{2\sqrt{3}}(x^2-y^2),\quad u_5=\frac{1}{6}(x^2+y^2-2z^2).
\end{eqnarray*}
Notice $(x,y,z)$ and $(-x,-y,-z)$ define the same point $(u_1,\ldots,u_5)$ on
the Veronese surface. To see it is an orbit of a group action, we need to change the standard inner product to the following one. Identify $\mathbb{R}^5$ with the space of real trace-free symmetric $3\times 3$-matrices, endowed with the inner product
$\langle A,B\rangle=\mathrm{Tr}AB$. This inner product is preserved by
all $\mathrm{SO}(3)$-conjugations. Denote $S^4(1)$ the unit sphere with respect to the new inner product. Then the conjugations of
$G=\mathrm{SO}(3)$ on $S^4(1)$ is of cohomogeneity one. Its stratification
consists of an open subset $N_1$ consisting of principal orbits, and two
closed $G$-orbits $N_2$ and $N_3$. The two singular orbits are just the
Veronese surfaces. By Corollary \ref{corollary-1}, we immediately see that they are Willmore orbits.

Theorem \ref{main-thm} predicts the existence of a principal Willmore orbit, which can be found the following observation.
For any symmetric $3\times 3$-matrix $A$, we denote its three eigenvalues as
$\lambda_1(A)\leq\lambda_2(A)\leq\lambda_3(A)$. Then all the principal orbits
$M_t$ in $N_1$
can be parametrized by the value of $\lambda_2(\cdot)$, i.e. $t=\lambda_2(\cdot)\in(-1,1)$. The antipodal map on
$S^4$ preserves $W(\cdot)$ and exchange $M_t$ with $M_{-t}$. So $t=0$ is
a critical point for $f(t)=W([M_t])$ with $t\in(-1,1)$. So by Theorem
\ref{theorem-1}, $M_0$, the hypersurface of all symmetric matrices in $S^4(1)$ with vanishing determinants, is a Willmore orbit in $S^4(1)$.

{\bf Example 3.}
Consider the $\mathrm{Ad}$-action of $G=\mathrm{SO}(5)$ on its Lie algebra $\mathfrak{g}$ preserves a canonical inner product
$\langle\cdot,\cdot\rangle_{\mathrm{bi}}$. The unit sphere $S^9(1)$ of all
vectors $v\in\mathfrak{g}$ with $\langle v,v\rangle_{\mathrm{bi}}=1$ is
a cohomogeneity one space for this $\mathrm{SO}(5)$-action. The two singular orbits, i.e.
$\widetilde{\mathrm{G}}_2(\mathbb{R}^5)=
\mathrm{SO}(5)/\mathrm{SO}(2)\mathrm{SO}(3)$ and
$\mathbb{C}\mathrm{P}^3=\mathrm{SO}(5)/\mathrm{U}(2)$, are Willmore
submanifolds by Corollary \ref{corollary-1}. Theorem \ref{main-thm} predicts
the existence of a principal Willmore orbit.

{\bf Example 4.}
Consider the
 the canonical action of
$G=\mathrm{SO}(n_1+1)\times\cdots\times\mathrm{SO}(n_p+1)$ on $S^{n+p-1}(1)$ where $n_i>0$ for each $i$, and
$\sum_{i=1}^p n_i=n$. The orbit type stratification for the $G$-action on $N$ can be described as
following.
For each non-empty subset ${A}\subset\{1,\ldots,p\}$,
$N_{A}$ is the union of all orbits
$$M_{t_1,\ldots,t_p}=S^{n_1}(t_1)\times\cdots\times S^{n_p}(t_p),$$
where $\sum_{i=1}^p t_i^2=1$, and for each $i$, $t_i\geq 0$ and the equality happens iff $i\notin A$. Then $N=\coprod_{\emptyset\neq A\subset\{1,\ldots,p\}}N_A$ is the orbit type stratification for the $G$-action on $N$.
Obviously $N_A\subset\bar{N_B}$ iff $A\subset B$. In particular,
the open submanifold $N_{\{1,\ldots,p\}}$ of $N$ consists of all principal orbits
$M_{t_1,\ldots,t_p}$ with nonzero $t_1$, $\ldots$, $t_p$.

Denote $I_i$ and $\mathbf{n}_i$ the first fundamental form and the outgoing unit normal field of $S^{n_i}(1)$ in
$\mathbb{R}^{n_i+1}$ respectively. Then restricted to $M_{t_1,\ldots,t_p}$, the outgoing unit normal field of $S^{n+p-1}(1)$ in $\mathbb{R}^{n+p}$ can be presented
as $\mathbf{n}=\sum_{i=1}^p t_i\mathbf{n}_i$.

The first fundamental form of
$M_{t_1,\ldots,t_p}$ is
$I=t_1^2 I_1+\cdots t_p^2 I_p$.
Then second fundamental form $\widetilde{\mathrm{II}}$ of $M_{t_1,\ldots,t_p}$ in
$\mathbb{R}^{n+p}$ is given by
$\widetilde{\mathrm{II}}(\mathbf{n}_i)=t_i I_i$. The mean curvature $\widetilde{H}$ of $\widetilde{\mathrm{II}}$
satisfies $\widetilde{H}(\mathbf{n}_i)=\frac{n_i}{n}t_i^{-1}$. So with respect to the imbedding $M_{t_1,\ldots,t_p}\subset S^{n+p-1}(1)$,
\begin{eqnarray*}
S&=&\sum_{i=1}^p||\widetilde{\mathrm{II}}(\mathbf{n_i})||^2-
||\widetilde{\mathrm{II}}(\mathbf{n})||^2
=\sum_{i=1}^p n_i (t_i^{-2}-1),\\
||H||^2 &=& ||\widetilde{H}||^2-\widetilde{H}^2(\mathbf{n})
=\frac{1}{n^2}[\sum_{i=1}^p n_i^2 t_i^{-2}]-1.
\end{eqnarray*}
So we have
\begin{eqnarray}\label{0001}
f(t_1,\ldots,t_p)&=& W(M_{t_1,\ldots,t_p})\nonumber\\
&=&\int_{M_{t_1,\ldots,t_p}}(S-n||H||^2)^{\frac{n}{2}}d
\mathrm{vol}_{M_{t_1,\ldots,t_p}}\nonumber\\
&=& C\cdot \prod_{i=1}^p t_i^{n_i}\cdot
\left[\sum_{i=1}^p\frac{n_i(n-n_i)}{n}t_i^{-2}\right]^{
\frac{n}{2}},
\end{eqnarray}
where $C$ is a positive constant.

Notice $f(t_1,\ldots,t_n)$ is homogeneous function of degree 0.
So $M_{t_1,\ldots,t_p}\subset S^{n+p-1}(1)$ is a Willmore orbit iff
$(t_1,\ldots,t_p)$ is a critical point for $f(t_1,\ldots,t_p)$. The only critical point of $f$
satisfying $\sum_{i=1}^p t_i^2=1$ and $t_i>0$ for each $i$ is given by $t_i=\sqrt{\frac{n-n_i}{n(p-1)}}$ for
each $i$ \cite{Li2002}. So the only principal Willmore orbit is
\begin{equation}\label{0002}
S^{n_1}(\sqrt{\frac{n-n_1}{n(p-1)}})\times
\cdots\times S^{n_p}(\sqrt{\frac{n-n_p}{n(p-1)}}).
\end{equation}

Then we look for Willmore orbits in other $N_A$. The most obvious ones are
$N_A$ where $A$ is a one-element set, which must be Willmore orbits by Corollary \ref{corollary-1}. In general
the closure of any
$N_A$ is a unit sphere. So the above argument also proves there
exists a unique Willmore orbit in $N_A$. For example, the only Willmore orbit
in $N_{\{1,\ldots,k\}}$ is
\begin{eqnarray*}
& &S^{n_1}(\sqrt{\frac{n'-n_1}{n'(k-1)}})\times
\cdots\times S^{n_k}(\sqrt{\frac{n'-n_k}{n'(k-1)}})\times0\times\cdots\times0,
\end{eqnarray*}
where $n'=\sum_{i=1}^k n_i$.

To summarize, Theorem \ref{main-thm} provides sharp estimates for the number
of Willmore orbits on the unit spheres in Example 1 and Example 4, which can be shown by short and
explicit calculation. For the other two examples, we can apply Theorem 5.1 of
\cite{Li2001} to show there are exactly three Willmore orbits, i.e. the estimates are still sharp, which can not be seen from Theorem \ref{main-thm}. On the other hand, these estimates for the number of Willmore orbits are stable when the unit sphere metrics are equivalently perturbed.

Ming Xu\\
School of Mathematical Sciences,
Capital Normal University,
Beijing 100048, P. R. China,
Email:mgmgmgxu@163.com\bigskip
\\
Jifu Li\\
School of Science,
Tianjin University of Technology,
Tianjin 300384, P. R. CHINA,
Email:ljfanhan@126.com
\end{document}